\documentclass[10pt]{article}
\usepackage{geometry,xcolor}
\usepackage{comment,hyperref,slashed,tensor}

\usepackage{amsmath}
\usepackage[T1]{fontenc}
\usepackage[utf8]{inputenc}
\usepackage{appendix}

\usepackage{ mathrsfs }
\usepackage{amsmath, amsfonts, mathtools, amsthm, amssymb}

\mathtoolsset{showonlyrefs}
\usepackage{todonotes, tikz}
\usepackage{hyperref,  color}

\newtheorem{theorem}{Theorem}[section]
\newtheorem{lemma}[theorem]{Lemma}
\newtheorem{proposition}[theorem]{Proposition}

\newtheorem{definition}[theorem]{Definition}

\newtheorem{remark}[theorem]{Remark}

\def\XXint#1#2#3{{\setbox0=\hbox{$#1{#2#3}{\int}$ }
\vcenter{\hbox{$#2#3$ }}\kern-.6\wd0}}

\title{Remark on the Stability of Energy Maximizers for the 2D Euler equation on $\mathbb{T}^2$}
\author{Tarek M. Elgindi}
\date{\today}

\begin{document}

\maketitle

\begin{center}
\emph{Dedicated to Vladimir \v{S}ver\'ak on the occasion of his 65th birthday.}
\end{center}

\begin{abstract}
It is well-known that the first energy shell, \[\mathcal{S}_1^{c_0}:=\{\alpha \cos(x+\mu)+\beta\cos(y+\lambda): \alpha^2+\beta^2=c_0\,\, \&\,\, (\mu,\lambda)\in\mathbb{R}^2\}\] of solutions to the 2d Euler equation is Lyapunov stable on $\mathbb{T}^2$.  This is simply a consequence of the conservation of energy and enstrophy. Using the idea of Wirosoetisno and Shepherd \cite{WS}, which is to take advantage of conservation of a properly chosen Casimir, we give a simple and quantitative proof of the $L^2$ stability of single modes up to translation. In other words, each \[\mathcal{S}_1^{\alpha,\beta}:=\{\alpha \cos(x+\mu)+\beta\cos(y+\lambda): (\mu,\lambda)\in\mathbb{R}^2\}\] is Lyapunov stable. Interestingly, our estimates indicate that the extremal cases $\alpha=0,$ $\beta=0$, and $\alpha=\pm\beta$ may be markedly less stable than the others. 
\end{abstract}

\tableofcontents

\section{Introduction}

Recall the 2d Euler equation in vorticity form \begin{equation}\label{2dEulerI}\partial_t\omega+u\cdot\nabla\omega=0,\end{equation}
\begin{equation}\label{2dEulerII}u=\nabla^\perp(-\Delta)^{-1}\omega.\end{equation} In studying the dynamics of solutions to \eqref{2dEulerI}-\eqref{2dEulerII}, it is important to recall the conservation laws enjoyed by regular solutions. Indeed, on any smooth domain, we have that the energy
\begin{equation}\label{Energy}E(\omega):=\int |u|^2\end{equation} is conserved\footnote{It is customary to take the no-penetration boundary condition $u\cdot n=0$ on the boundary of the domain, though we will only be concerned with domains without boundary here.}. We also have the conservation of all Casimirs:
\begin{equation}\label{Casimirs}H_{f}(\omega):=\int f(\omega).\end{equation} The conservation of the Casimirs \eqref{Casimirs} is equivalent to the statement that any regular solution to the Euler equation is always a volume preserving rearrangement of its initial data. All of these conservation laws put infinitely many constraints on solutions to \eqref{2dEulerI}-\eqref{2dEulerII} and thus play a crucial role in describing their dynamics. 
In particular, they play a vital role in the study of (nonlinear) stability of steady states. 
 
The most powerful tool for studying nonlinear stability that we are aware of is Kelvin's variational principle \cite{K}, which was placed in a general setting by Arnold \cite{arnold,AK} in the 1960's. Arnold's theory of stability, which is expounded upon extensively in the notes of V. \v{S}ver\'ak \cite{SverakNotes}, is based on a simple minimization principle using the conserved quantities. A simple example of this principle is 
\begin{lemma}
Consider the ordinary differential equation on $\mathbb{R}^d:$
\begin{equation}\label{ODE}\frac{d}{dt}x = N(x),\end{equation} for some smooth $N:\mathbb{R}^d\rightarrow\mathbb{R}^d.$ Assume that $N$ has a first integral $E:\mathbb{R}^d\rightarrow\mathbb{R},$ in other words that $\frac{d}{dt} E(x)=0,$ for every solution $x$ to \eqref{ODE}. If $x_*$ is a strict local extremizer of $E,$ then $x_*$ is a Lyapunov stable steady solution of \eqref{ODE}.
\end{lemma}
The proof is elementary. 
Note that this result can be extended to the case when \eqref{ODE} has other first integrals $H_i$ with the weaker assumption that $x_*$ is an extremizer of $E$ on a single leaf $\cap_{i}\{H_i=c_i\}$ (see \cite[Theorem 3.3]{AK}).  For the Euler equation, the energy \eqref{Energy} is a first integral; moreover, a single solution is always a rearrangement of its initial data. This motivates Arnold's notion of hydrodynamic stability\footnote{In fact, we are giving a more general notion than the one given in Arnold's original paper that was given by Burton in \cite{Burton} (it can be argued that this concept is even present in Kelvin's short note \cite{K}).}: 
\begin{definition}
Fix a two dimensional domain $\Omega.$ We say that $\omega_*$ is an Arnold-stable 2d Euler steady state if it is a strict local extremizer of $E$ among all volume preserving rearrangements of $\omega_*.$ 
\end{definition}
Arnold-stable steady states exist on all simply connected domains; in fact, one can construct an infinite-dimensional family of Arnold-stable steady states on any simply connected domain \cite{CS, CDG}. See also \cite{GS} for recent advances on Arnold stable solutions on $\mathbb{R}^2$ and the vanishing viscosity limit. On spatially periodic domains (like rectangular or square tori), the situation is quite different. It is not difficult to see that there are no non-trivial Arnold stable steady states on $\mathbb{T}^2$, due to translation invariance. In fact, the author is unaware of \emph{any} non-trivial Lyapunov stable steady state on $\mathbb{T}^2$. The purpose of this work is to review a part of the picture on steady states on $\mathbb{T}^2$ of particular interest: the global maxima of energy.

\subsection{Main Theorem}
If we restrict our attention to vorticities with unity $L^2$ norm, we find that any element of the three-dimensional manifold \[\mathcal{S}^1_{1}=\{\alpha \cos(x+\mu)+\beta \cos(y+\lambda): \alpha^2+\beta^2=1, (\mu,\lambda)\in\mathbb{R}^2\}\] is a global maximizer of energy. Since all of these are obviously not \emph{strict} global maximizers, we just get that $\mathcal{S}^1_1$ itself is $L^2$ stable. Whether any particular element of $\mathcal{S}^1_1$ is stable is an open problem; however, it turns out that one can use conservation of the Casimirs to slightly restrict the dynamics.  Now we state our main theorem regarding the sets $\mathcal{S}_1^{\alpha,\beta}:$
\[\mathcal{S}_1^{\alpha,\beta}=\{\alpha \cos(x+\mu)+\beta\cos(y+\lambda): (\mu,\lambda)\in\mathbb{R}^2\}.\] For ease of exposition, we take $\alpha^2+\beta^2=1$.
\begin{theorem}
\label{MainTheorem}
 Define the "extremal set" $\mathcal{E}:=\{(\pm 1,0),(0,\pm 1),(\pm\frac{1}{\sqrt{2}},\pm\frac{1}{\sqrt{2}})\},$ consisting of eight points on the unit circle. If $(\alpha,\beta)\in \mathcal{E},$ there exists $C>0$ so that for all $\epsilon$ sufficiently small, we have that
\[d(\omega_0,\mathcal{S}_1^{\alpha,\beta})<\epsilon \implies d(\omega(t), \mathcal{S}_1^{\alpha,\beta})<C\sqrt{\epsilon},\] for all $t\in\mathbb{R}.$ In contrast, for each $(\alpha,\beta)\in \{\alpha^2+\beta^2=1\}\setminus \mathcal{E},$ there exists a constant $C>0$ so that for all $\epsilon$ sufficiently small, we have that  
\[d(\omega_0,\mathcal{S}_1^{\alpha,\beta})<\epsilon \implies d(\omega(t), \mathcal{S}_1^{\alpha,\beta})<C \epsilon,\]
for all $t\in\mathbb{R}.$
\end{theorem}

\begin{remark}
As examples, the steady states $\cos(y)$ and $\frac{1}{\sqrt{2}}(\cos(x)+\cos(y))$ are extremal while the steady state $\frac{1}{2}\cos(x)+\frac{\sqrt{3}}{2}\cos(y)$ is not extremal (thus possibly more stable).
\end{remark}

\subsection{Remarks on the Main Theorem}
The idea of the proof originates in the paper of Wirosoetisno and Shepherd \cite{WS}, which is to combine the original argument of Arnold with the conservation of higher norms of vorticity (to distinguish between different points on the first shell). The argument of \cite{WS} does not imply $L^2$ stability because of its use of higher order norms. The stability argument here is a bit different and we also make use of a different conserved quantity that is well-defined for $L^2$ solutions. We should remark that, very recently, the authors of \cite{WZ} established a \emph{non-quantitative} stability result based again on the idea of \cite{WS}. The argument of \cite{WZ} is by contradiction and uses the whole transport of vorticity rather than just a single Casimir (in the spirit of Burton's work \cite{Burton}). One advantage of the argument given here is that we give an elementary proof along with quantitative bounds and also provide evidence that different elements of the first energy shell may be more stable than others (i.e. extremal vs. non-extremal values of $(\alpha,\beta)$ in Theorem \ref{MainTheorem}). In particular, exact shear flows and exact cellular flows on the first shell appear to be less stable than those flows with regions of both shearing as well as so-called cat's-eye structures.

It would be very interesting to determine whether the first stability bound in Theorem \ref{MainTheorem} is actually sharp. One potential explanation for the weaker stability of extremal points, which was offered by T. Drivas, is that if one varies $(\alpha,\beta)$ continuously, the extremal points are precisely the points at which there is a break in the topology of the streamlines. This implies that the foliation of the space of vorticities by the isovortical leaves may be degenerate precisely at the extremal points. To further investigate whether extremal points truly enjoy different stability properties, it would be interesting to study the structure of the set of steady states close to extremal and non-extremal points. A strong degeneracy was shown to be present near the shear flows on the first shell (which are extremal) in \cite{CZEWII}. For comparison, it would be good to study the non-extremal case; perhaps there one can establish a result in the spirit of \cite{CS}. It would also be very interesting to know whether \emph{individual} elements of the first shell are stable (i.e. taking the velocity to have zero-mean and removing the translation); perhaps the ideas of \cite{DEJ} could be applied to that problem. Let us remark finally that the stability of the first Fourier shell is relevant also in the analysis of the long-time behavior of solutions to the Navier-Stokes equations \cite{CTV}. In fact, there is a recent interesting numerical computation, done by T. Drivas, of the stochastically forced Navier-Stokes system that shows the solution travelling on the first shell \url{https://www.youtube.com/watch?v=8H4Xee6-_7g}. The computation qualitatively shows that the exact shear states do not appear to be as stable as states with cat's eyes.

\section{Proof of the Main Theorem}
The basic idea of the proof is to first recall that the whole of the first shell, $\mathcal{S}^{1}_1$ is $L^2$ stable. This follows from the conservation of energy and enstrophy. This means that in order to establish the stability of $S^{\alpha,\beta}_1$, we need only show that there exists an invariant of the Euler equation, $I:L^2\rightarrow\mathbb{R},$ with the property that 
$I(\mathcal{S}^{\alpha,\beta})$ (which is in fact only a single number that depends on $(\alpha,\beta)$) varies as we vary $(\alpha,\beta).$ Conservation of $I$ will then imply the all-time stability of $\mathcal{S}^{\alpha,\beta}_1$. This can be made quantitative by showing that the derivative of the value of $I(\mathcal{S}^{\alpha,\beta})$ along the circle $\alpha^2+\beta^2=1$ either doesn't vanish (which is the case at the non-extremal points) or vanishes but only to first order (which is the case at extremal points).

\subsection{$L^2$ Stability of the First Shell}\label{FirstShell}
We start by showing the stability of $\mathcal{S}^1_1.$
Let us begin with some notation.  We denote by \[\mathbb{Z}^2_0:=\mathbb{Z}^2\setminus\{(0,0)\}.\] We define $\mathbb{P}$ to be the orthogonal $L^2$ projector onto $\mathcal{S}_1$ and denote by \[\mathbb{P}^\perp=\text{Id}-\mathbb{P}.\] 
\begin{proposition}
Let $\omega$ be a smooth solution to the 2d Euler equation. Then, 
\[|\mathbb{P}^\perp\omega|_{L^2}\leq \sqrt{2}|\mathbb{P}^\perp\omega_0|_{L^2}.\]
\end{proposition}
\begin{proof}
We have that 
\[\frac{1}{2}|\mathbb{P}^\perp\omega|_{L^2}^2\leq \sum_{k\in\mathbb{Z}^2_0}(1-\frac{1}{|k|^2})|\hat{\omega}(k)|^2=|\omega|_{L^2}^2-|u|_{L^2}^2=|\omega_0|_{L^2}^2-|u_0|_{L^2}^2=\sum_{k\in\mathbb{Z}^2_0}(1-\frac{1}{|k|^2})|\hat{\omega}_0(k)|^2\leq |\mathbb{P}^\perp\omega_0|_{L^2}^2,\] where we used on both sides that the multiplier is zero when $|k|=1.$
\end{proof}

\subsection{Two calculus lemmas}
We will need two simple calculus lemmas. First, let us state the one we will apply at non-extremal points:
\begin{lemma}\label{CalculusLemmaI}
Assume $f:\mathbb{R}\rightarrow\mathbb{S}^1$ is continuous. Assume that $F:\mathbb{S}^1\rightarrow\mathbb{R}$ is smooth and that $F'(\theta_0)\not=0.$ 
Then, there exists a fixed $C:=C(F,\theta_0)>0$ so that if $\epsilon>0$ is sufficiently small and 
\[|F(f(t))-F(\theta_0)|+|f(0)-\theta_0|<\epsilon,\]  then we have that
\[|f(t)-\theta_0|<C\epsilon,\] for all $t\in\mathbb{R}.$
\end{lemma}
The second lemma will be applied at extremal points:
\begin{lemma}\label{CalculusLemmaII}
Under the same assumptions as Lemma \ref{CalculusLemmaI}, except that $F'(\theta_0)=0,$ $F''(\theta_0)\not=0,$ we conclude that \[|f(t)-\theta_0|<C\sqrt{\epsilon},\] for all $t\in\mathbb{R}.$
\end{lemma}
Both lemmas are elementary and follow from Taylor expanding $F$ around $\theta_0$ and using the continuity of $f.$

\subsection{A conserved quantity that distinguishes values of $(\alpha,\beta)$}

The reason that the standard Arnold method, which here would consist of maximizing the energy for fixed enstrophy, only gives stability of $\mathcal{S}^1_1$ and not $\mathcal{S}^{\alpha,\beta}_1$ is that both the energy and the enstrophy are constant on $\mathcal{S}^{\alpha,\beta}_1$. In particular, one cannot distinguish the sets $\mathcal{S}^{\alpha,\beta}_1$ using only the knowledge that the enstrophy and energy are conserved. As brilliantly observed by Wirosoetisno and Shepherd in \cite{WS}, we can distinguish the various values of $(\alpha,\beta)$ using a higher order Casimir, like the $L^4$ norm of vorticity. Unfortunately, using the $L^4$ norm will not allow us to deduce stability in $L^2$. We use a small modification of the $L^4$ norm with a conserved quantity that is not technically a Casimir. 
Indeed, for any measurable function $\omega\in \dot{H}^{-1}$, we may define the following quantity:
\[I(\omega)=\int |\omega|^4\chi\Big(\frac{\omega}{10|u|_{L^2}}\Big),\] where $\chi\in C^\infty(\mathbb{R})$ is even and equal to 1 on $[0,1]$ and equal to 0 on $[2,\infty).$
When $\omega$ solves the 2d Euler equation, we of course have that \[I(\omega)=I(\omega_0).\]
Let us now define:
\[\mathcal{S}_{*}^1=\{\alpha\cos(x)+\beta\cos(y): \alpha^2+\beta^2=1\}.\] Now writing $(\alpha,\beta)=(\cos(\theta),\sin(\theta))$, we may define \[F(\theta)=I(\cos(\theta) \cos(x)+\sin(\theta)\cos(y)),\] for $\theta\in \mathbb{S}^1.$
The key observation is the following.

\begin{lemma}\label{J0Lemma}
$F'(\theta)=0$ if and only if $\theta=\frac{\pi}{4}k,$ $k\in\mathbb{Z}.$ Moreover,  $F''(\frac{\pi}{4}k)\not=0,$ for $k\in\mathbb{Z}.$
\end{lemma}
\begin{proof}
By definition of $I,$ we have that 
\begin{align}
    F(\theta)&=\int_{\mathbb{T}^2}(\cos(\theta)\cos(x)+\sin(\theta)\cos(y))^4dxdy\\  &=2\pi\cos^4(\theta)\int_{0}^{2\pi}\cos^4(x)dx+6\pi^2\cos^2(\theta)\sin^2(\theta)+2\pi\sin^4(\theta)\int_0^{2\pi}\cos^4(y)dy\\
    &=\frac{3\pi^2}{2}(\cos^4(\theta)+4\cos^2(\theta)\sin^2(\theta)+\sin^4(\theta))=\frac{3\pi^2}{2}(1+2\cos^2(\theta)-2\cos^4(\theta))\\
    &=\frac{3\pi^2}{2}(\frac{5}{4}-\frac{1}{4}\cos(4\theta)).
\end{align}
We thus see that $F'(\theta)=\frac{3\pi^2}{2}\sin(4\theta),$ from which the result follows. 

\end{proof}

Next, we define $\mathbb{P}_*:\mathcal{S}_1\rightarrow\mathcal{S}_*$ by 
\[\mathbb{P}_*(\alpha\cos(x+\mu)+\beta\cos(y+\lambda))=\alpha\cos(x)+\beta\cos(y).\]

Now we see that Lemma \ref{J0Lemma} combined with the calculus Lemmas \ref{CalculusLemmaI}-\ref{CalculusLemmaII} below easily imply the Main Theorem \ref{MainTheorem}. Indeed, fix $(\alpha,\beta)$ and suppose that $\omega_0$ is smooth and that $dist(\omega_0, \mathcal{S}^{\alpha,\beta}_1)<\epsilon.$ It follows that 
\[|\mathbb{P}^\perp\omega_0|_{L^2}\leq \sqrt{2}|\mathbb{P}^\perp\omega_0|<\sqrt{2}\epsilon.\] 
It follows that 
\[|I(\omega)-I(\mathbb{P}\omega)|\leq C|\mathbb{P}^\perp\omega|_{L^2}\leq C\epsilon,\] since $I$ is clearly Lipschitz continuous on $L^2.$ Now, by translation invariance of $I$, it follows that  \[|I(\omega)-I(\mathbb{P}_*(\omega))|\leq C\epsilon.\] Since $I(\omega)=I(\omega_0),$ it follows (from two applications of the inequality) that
\[|I(\mathbb{P}_*\omega_0)-I(\mathbb{P}_*\omega)|\leq C\epsilon.\] Conservation of the $L^2$ norm and Lemmas \ref{J0Lemma},\ref{CalculusLemmaI}, and \ref{CalculusLemmaII} now give the result:
\[|\mathbb{P}_*\omega-\mathbb{P}_*\omega_0|_{L^2}\leq C\epsilon,\] when $(\alpha,\beta)$ is not extremal, while
\[|\mathbb{P}_*\omega-\mathbb{P}_*\omega_0|_{L^2}\leq C\sqrt{\epsilon},\] when $(\alpha,\beta)$ is extremal, so long as $\epsilon$ is sufficiently small.  This concludes the proof of the Main Theorem \ref{MainTheorem}.

\begin{remark}
Let us close by noting that the reason that the extremal points $\mathcal{E}$ enjoy weaker stability estimates is that they are critical points of the Casimir we have chosen when restricted to $\mathcal{S}_*$. It can be checked directly that the criticality of the extremal points is independent of the choice of the Casimir (as long as it is smooth). This indicates that the extremal points may really be special, though we have no proof that the first estimate in Theorem \ref{MainTheorem} is actually sharp.  
\end{remark}

\section{Acknowledgements}
The author thanks V. \v{S}ver\'ak  for his guidance and many helpful discussions over the years that helped shape the author's view of fluid mechanics and PDE. The author also thanks T. Drivas for multiple comments that improved this paper. He finally acknowledges funding from the NSF DMS-2043024 and the Alfred P. Sloan foundation.


\begin{thebibliography}{99}


 \bibitem{arnold}
 \newblock V.Arnold,
 \newblock \emph{Sur la g\'eom\'etrie diff\'erentielle des groupes de Lie de
 dimension infinie et ses applications \`a l'hydrodynamique des
 fluides parfaits},
 \newblock Ann. Inst. Fourier (Grenoble), \textbf{16} (1966),
fasc. 1, 319--361.

\bibitem{AK} 
Arnol'd, Vladimir Igorevich, and Boris A. Khesin. "Topological methods in hydrodynamics." Annual review of fluid mechanics 24.1 (1992): 145-166.

\bibitem{Burton}
Burton, G. R. "Global nonlinear stability for steady ideal fluid flow in bounded planar domains." Archive for rational mechanics and analysis 176 (2005): 149-163.

\bibitem{CS} 
A. Choffrut, and V. Šverák. "Local structure of the set of steady-state solutions to the 2D incompressible Euler equations." Geometric and Functional Analysis 22.1 (2012): 136-201.

\bibitem{CDG}
Constantin, Peter, Theodore D. Drivas, and Daniel Ginsberg. "Flexibility and rigidity in steady fluid motion." Communications in Mathematical Physics 385 (2021): 521-563.

\bibitem{CTV}
Constantin, Peter, Andrei Tarfulea, and Vlad Vicol. "Absence of anomalous dissipation of energy in forced two dimensional fluid equations." Archive for Rational Mechanics and Analysis 212 (2014): 875-903.

\bibitem{CZEWII}
Coti Zelati, Michele, Tarek M. Elgindi, and Klaus Widmayer. "Stationary structures near the Kolmogorov and Poiseuille flows in the 2d Euler equations." Archive for Rational Mechanics and Analysis 247.1 (2023): 12.


\bibitem{DE}
Drivas, Theodore D., and Tarek M. Elgindi. "Singularity formation in the incompressible Euler equation in finite and infinite time." arXiv preprint arXiv:2203.17221 (2022).

\bibitem{DEJ}
Drivas, Theodore D., Tarek M. Elgindi, and In-Jee Jeong. "Twisting in Hamiltonian Flows and Perfect Fluids." arXiv preprint arXiv:2305.09582 (2023).

\bibitem{GS}
Gallay, Thierry, and Vladimir Sverak. "Arnold's variational principle and its application to the stability of planar vortices." arXiv preprint arXiv:2110.13739 (2021).

\bibitem{MS}
Meshalkin, L. D., and Ia G. Sinai. "Investigation of the stability of a stationary solution of a system of equations for the plane movement of an incompressible viscous liquid." Journal of Applied Mathematics and Mechanics 25.6 (1961): 1700-1705.

\bibitem{SverakNotes}
V. \v{S}ver\'ak. "Selected Topics in Fluid Mechanics," 2011.

\bibitem{K}
Thomson, William. "Maximum and minimum energy in vortex motion." Nature 22 (1880): 618-620.

\bibitem{WZ}
Wang, Guodong, and Bijun Zuo. "Nonlinear stability of sinusoidal Euler flows on a flat torus." arXiv preprint arXiv:2210.01405 (2022).

\bibitem{WS}
Wirosoetisno, Djoko, and Theodore G. Shepherd. "Nonlinear stability of Euler flows in two-dimensional periodic domains." Geophysical \& Astrophysical Fluid Dynamics 90.3-4 (1999): 229-246.



\end{thebibliography}
\end{document}